\documentclass[10pt]{article}
\usepackage{amsfonts}
\usepackage{amsmath}
\usepackage{harvard}
\usepackage[hypertex]{hyperref}

\newtheorem{theorem}{Theorem}

\newtheorem{corollary}[theorem]{Corollary}

\newtheorem{lemma}[theorem]{Lemma}

\input{tcilatex}

\hypersetup{colorlinks,citecolor=red,pdfstartview=FitH, pdfpagemode=None}

\begin{document}

\title{Determinants and Inverses of Circulant Matrices with Jacobsthal and
Jacobsthal-Lucas Numbers}
\author{Durmu\c{s} Bozkurt\thanks{
e-mail: dbozkurt@selcuk.edu.tr}\ {and} Tin-Yau Tam\thanks{%
e-mail: tamtiny@auburn.edu} \and Department of Mathematics, Sel\c{c}uk
University, \\
{42075 Kampus, Konya, Turkey } \and Department of Mathematics and {Statistics%
}, Auburn University, \\
{AL 36849-5310, USA}}
\maketitle

\begin{abstract}
Let $n\geq 3$ and $\mathbb{J}_{n}:=\mbox{circ}(J_{1},J_{2},\ldots ,J_{n})$
and $\mathbb{j} _{n}:=\mbox{circ}(j_{0},j_{1},\ldots , {j}_{n-1})$ be {the} $%
n\times n$ circulant matrices{,} associated with {the} $n$th Jacobsthal {%
number} $J_{n}$ and {the $n$th} Jacobsthal-Lucas number $j_{n}$,
respectively. The determinants of $\mathbb{J}_{n}$\ and $\mathbb{j}_{n}$\ {%
are obtained in terms of} the Jacobsthal and Jacobsthal-Lucas numbers. {\
These imply that} $\mathbb{J}_{n}$\ and $\mathbb{j}_{n}$\ are invertible. {%
We also derive the} inverses of $\mathbb{J}_{n}$ and $\mathbb{j}_{n}$.
\end{abstract}

\section{Introduction}

{The $n\times n$ circulant matrix} {$C_n:=\mbox{circ}(c_{0},c_{1},\ldots
,c_{n-1})$, assoicated with the numbers $c_0, \dots, c_{n-1}$,} is defined
as 
\begin{equation}
C_n{:=}\left[ 
\begin{array}{ccccc}
c_{0} & c_{1} & \ldots & c_{n-2} & c_{n-1} \\ 
c_{n-1} & c_{0} & \ldots & c_{n-3} & c_{n-2} \\ 
\vdots & \vdots & \ddots & \vdots & \vdots \\ 
c_{2} & c_{3} & \ldots & c_{0} & c_{1} \\ 
c_{1} & c_{2} & \ldots & c_{n-1} & c_{0}%
\end{array}%
\right].  \label{2}
\end{equation}%
%
%
%
%
%
Circulant matrices have a {wide} range of {applications, for examples} in
signal processing, coding theory, image processing, digital image disposal,
self-regress {\ design} and so on. {Numerical} solutions of the certain
types of elliptic and parabolic partial differential equations with periodic
boundary conditions often involve linear systems {associated with circulant
matrices} [9-11].

\textcolor[rgb]{0.93,0.08,0.17}{The eigenvalues and eigenvectors of $C_n$ are well-known [14]: $$\lambda_j =
\sum_{k=0}^{n-1}c_{k}\omega^{jk}, \quad j=0, \dots, n-1,
$$ where
$\omega {:=}\exp (\frac{2\pi i}{n})$ and $i:=\sqrt{-1}$ and
the corresponding  eigenvectors
$$
v_j = (1, \omega^j, \omega^{2j}, \dots, \omega^{(n-1)j})^T,\quad j=0, \dots, n-1.$$ Thus}
we have the determinants and inverses of nonsingular circulant matrices
[1,3,4,14]{:} 
\begin{equation*}
\det (C_{n})=\textcolor[rgb]{0.93,0.08,0.17}{
\prod_{j=0}^{n-1}(\sum_{k=0}^{n-1}c_{k}\omega^{jk})} {,}
\end{equation*}%
%
%
and%
\begin{equation*}
C_{n}^{-1}=\mbox{circ}(a_{0},a_{1},\ldots ,a_{n-1}){,}
\end{equation*}%
%
%
%
%
%
where 
\textcolor[rgb]{0.93,0.08,0.17}{$a_{j}{:=}\frac{1}{n}
\sum_{k=0}^{n-1} \lambda_k \omega^{-kj}$,} and $r=0,1,\dots ,n-1$ [4]. When {%
$n$ is getting large}, the above formulas {are not very handy to use}. 
\textcolor[rgb]{0.93,0.08,0.17}{If there is some structure among $c_0, \dots, c_{n-1}$, we may be able to get
more explicit forms of the eigenvalues, determinants and inverses of $C_n$.}
Recently, {studies} on the circulant matrices involving {interesting} number
sequences {appeared}. In [1] the determinants and inverses of the circulant
matrices $\mathbb{A}_{n}=\mbox{circ}(F_{1},F_{2},\ldots ,F_{n})$ and $%
\mathbb{B}_{n}=\mbox{circ}(L_{1},L_{2},\ldots ,L_{n})$ {are} derived{,}
where $F_{n}$ and $L_{n}$ {are the} $n$th Fibonacci and Lucas numbers,
respectively. In [2] the $r$-circulant matrix {is} defined and its {norm is}
computed. {The} norms of Toeplitz matrices [13] involving Fibonacci and
Lucas numbers {are obtained} [5]. Miladinovic and Stanimirovic [6] {gave} an
explicit formula of the Moore-Penrose inverse of singular generalized
Fibonacci matrix. Lee and et al. found the factorizations and eigenvalues of
Fibonacci and symmetric Fibonacci matrices [7].

{When} $n\geq 2,$ the Jacobsthal and Jacobsthal-Lucas sequences $\{J_{n}\}$
and $\{j_{n}\}$ are defined by $J_{n}=J_{n-1}+2J_{n-2}$ and $%
j_{n}=j_{n-1}+2j_{n-2}$\ with initial conditions $J_{0}=0${,} $J_{1}=1${,} $%
j_{0}=2${,} and $j_{1}=1${,} respectively. Let $\mathbb{J}_{n}:=\mbox{circ}%
(J_{1},J_{2},\ldots ,J_{n})$ and $\mathbb{j} _{n}:=\mbox{circ}%
(j_{0},j_{1},\ldots , {j}_{n-1})$. 
The aim of this paper is to establish {\ some useful} formulas for the
determinants and inverses of $\mathbb{J}_{n}$ and $\mathbb{j}_{n}$ using the
nice properties of the Jacobsthal and Jacobsthal-Lucas numbers. \textbf{%
Question: How about eigenvalues?} {Matrix decompositions are derived for $%
\mathbb{J}_{n}$ and $\mathbb{j}_{n}$ in order to obtain the results. }


\section{Determinants of {$\mathbb{J}_{n}$ and $\mathbb{j}_{n}$}}


{Recall that} $\mathbb{J}_{n}:=\mbox{circ}(J_{1},J_{2},\ldots ,J_{n})$ and $%
\mathbb{j}_{n}:=\mbox{circ}(j_{0},j_{1},\ldots ,{j}_{n-1})${, i.e.,} 
where $J_{k}$ and $j_{k}$ are {the} $k$th Jacobsthal and
Jacobsthal-Lucas numbers{, respectively,} with the {recurrence} relations $%
J_{k}=J_{k-1}+2J_{k-2}${,} $j_{k}=j_{k-1}+2j_{k-2}${, and} the initial
conditions $J_{0}=0${,} $J_{1}=1${,} $\ j_{0}=2${,} and $j_{1}=1\ (k\geq 2)$%
. Let $\alpha $ and $\beta $ be the roots of $x^{2}-x-2=0.$ {Using the}
Binet formulas [8, p.40] {for} the sequences $\{J_{n}\}$ and $\{j_{n}\} $, {%
one has} 
\begin{equation}
J_{n}=\frac{\alpha ^{n}-\beta ^{n}}{3}=\frac{1}{3}[2^{n}-(-1)^{n}]  \label{1}
\end{equation}%
and 
\begin{equation}
j_{n}=\alpha ^{n}+\beta ^{n}=2^{n}+(-1)^{n}.  \label{20}
\end{equation}%
%
%
%

\begin{theorem}
\label{t1} {Let $n\geq 3$}. Then%
\begin{equation}
\det (\mathbb{J}_{n})=(1-J_{n+1})^{n-2}(1-J_{n})+2\sum_{k=1}^{n-2}\left[
J_{k}(1-J_{n+1})^{k-1}(2J_{n})^{n-k-1}\right]{.}  \label{4}
\end{equation}%
%
%
%
%
%
\end{theorem}

\begin{proof}
Obviously, $\det (\mathbb{J}_{3})=20.$ It satisfies 
\eqref{4}.
For $n>3,$ we select the matrices $P_{n}$ and $Q_{n}$ so that when we multiply
$\mathbb{J}_{n}$ with $P_{n}$ on the left and $Q_{n}$ on the right we obtain a special upper triangular matrix
that have nonzero entries only on the first two rows, main diagonal and super diagonal:%
\begin{equation}
P_{n}{:=}\begin{bmatrix}
1 & 0 & 0 & 0 & \ldots & 0 & 0 \\
-1 & 0 & 0 & 0 & \ldots & 0 & 1 \\
-2 & 0 & 0 & 0 & \ldots & 1 & -1 \\
0 & 0 & 0 & 0 & \ldots & -1 & -2 \\
\vdots & \vdots & \vdots & \vdots &  & \vdots & \vdots \\
0 & 0 & 1 & -1 & \ldots & 0 & 0 \\
0 & 1 & -1 & -2 & \ldots & 0 & 0%
\end{bmatrix}  \label{5}
\end{equation}%
and%
\begin{equation*}
Q_{n}{:=}\begin{bmatrix}
1 & 0 & 0 & \ldots & 0 & 0 \\
0 & \left( \frac{2J_{n}}{1-J_{n+1}}\right) ^{n-2} & 0 & \ldots & 0 & 0 \\
0 & \left( \frac{2J_{n}}{1-J_{n+1}}\right) ^{n-3} & 0 & \ldots & 0 & -1 \\
0 & \left( \frac{2J_{n}}{1-J_{n+1}}\right) ^{n-4} & 0 & \ldots & -1 & 0 \\
\vdots & \vdots & \vdots &  & \vdots & \vdots \\
0 & \left( \frac{2J_{n}}{1-J_{n+1}}\right) & -1 & \ldots & 0 & 0 \\
0 & 1 & 0 & \ldots & 0 & 0%
\end{bmatrix}.
\end{equation*}%
{Notice that we have the following equivalence:}
\begin{eqnarray*}
S_{n} &=&P_{n}\mathbb{J}_{n}Q_{n} \\
&=&\begin{bmatrix}
1 & f_{n} & -J_{n} & -J_{n-1} & -J_{n-2} & \ldots & -J_{4} & -J_{3} \\
& g_{n} & J_{n}-1 & -2J_{n-2} & -2J_{n-3} & \ldots & -2J_{3} & -2J_{2} \\
&  & 2J_{n} & J_{n+1}-1 &  &  &  &  \\
&  &  & 2J_{n} & J_{n+1}-1 &  &  &  \\
&  &  &  &  &  &  & 0 \\
&  &  &  & \ddots & \ddots &  &  \\
&  &  &  &  &  & J_{n+1}-1 &  \\
&  &  &  &  & 2J_{n} &  &  \\
&  &  &  &  &  & 2J_{n} & J_{n+1}-1 \\
&  &  &  &  &  &  & 2J_{n}%
\end{bmatrix}
\end{eqnarray*}%
{and $S_n$} is upper triangular, where
\begin{eqnarray*}
f_{n}&{:=}&\sum _{k=1}^{n-1}J_{k+1}\left( \frac{2J_{n}}{1-J_{n+1}}\right) ^{n-k-1}{,}
\\
g_{n}&{:=}&1-J_{n}+2\sum_{k=1}^{n-2}J_{n-k-1}\left( \frac{2J_{n}}{1-J_{n+1}}\right) ^{k}.
\end{eqnarray*} Then we have%
\begin{equation*}
\det (S_{n})=\det (P_{n})\det (\mathbb{J}_{n})\det
(Q_{n})=(2J_{n})^{n-2}g_{n}.
\end{equation*}%
Since%
\begin{equation*}
\det (P_{n})=
\begin{cases} 1& n\equiv 1,  \mbox{ or }  2\ (\mbox{mod}4) \\
-1, & n\equiv 0\ \mbox{ or } 3\ (\mbox{mod}4){,}%
\end{cases}
\end{equation*}%
and%
\begin{equation*}
\det (Q_{n})=
\begin{cases}
\left( \frac{2J_{n}}{1-J_{n+1}}\right) ^{n-2}, &  n\equiv 1 \mbox{ or } 2\
(\mbox{mod}4) \\
-\left( \frac{2J_{n}}{1-J_{n+1}}\right) ^{n-2},& n\equiv 0\mbox{ or } 3\ (\mbox{mod}%
4){,}%
\end{cases}
\end{equation*}%
for all $n>3,$%
\begin{equation*}
\det (P_{n})\det (Q_{n})=\left( \frac{2J_{n}}{1-J_{n+1}}\right) ^{n-2}
\end{equation*}%
{and} \eqref{4} follows.
\end{proof}

\begin{theorem}
\label{t2} {Let $n\geq 3$}. Then%
\begin{equation}
\det (\mathbb{j}_{n})=(2-j_{n})^{n-2}(4-j_{n-1})+\sum_{k=2}^{n-1}\left[
(2j_{k}-j_{k-1})(2-j_{n})^{k-2}(1+2j_{n-1})^{n-k}\right].  \label{60}
\end{equation}%
%
%
%
%
%
\end{theorem}

\begin{proof}
Since $\det (\mathbb{j}_{3})=104,$ {$\mathbb{j}_{3}$} satisfies (\ref{60}). For $n>3,$ we select the matrices $K_{n}$ and $M_{n}$ so that when
we multiply $\mathbb{j}_{n}$ with $K_{n}$ on the left and $M_{n}$ on the right we obtain a special upper triangular matrix
that have nonzero entries only on the first two rows, main diagonal and super diagonal:%
\begin{equation}
K_{n} {:=}\begin{bmatrix}
1 & 0 & 0 & 0 & \ldots & 0 & 0 \\
-\frac{1}{2} & 0 & 0 & 0 & \ldots & 0 & 1 \\
-2 & 0 & 0 & 0 & \ldots & 1 & -1 \\
0 & 0 & 0 & 0 & \ldots & -1 & -2 \\
\vdots & \vdots & \vdots & \vdots & \ddots & \vdots & \vdots \\
0 & 0 & 1 & -1 & \ldots & 0 & 0 \\
0 & 1 & -1 & -2 & \ldots & 0 & 0%
\end{bmatrix}  \label{50}
\end{equation}%
and%
\begin{equation*}
M_{n}{:=}\begin{bmatrix}
1 & 0 & 0 & \ldots & 0 & 0 \\
0 & \left( \frac{1+2j_{n-1}}{2-j_{n}}\right) ^{n-2} & 0 & \ldots & 0 & 0 \\
0 & \left( \frac{1+2j_{n-1}}{2-j_{n}}\right) ^{n-3} & 0 & \ldots & 0 & -1 \\
0 & \left( \frac{1+2j_{n-1}}{2-j_{n}}\right) ^{n-4} & 0 & \ldots & -1 & 0 \\
\vdots & \vdots & \vdots &  & \vdots & \vdots \\
0 & \left( \frac{1+2j_{n-1}}{2-j_{n}}\right) & 0 & \ldots & 0 & 0 \\
0 & 1 & -1 & \ldots & 0 & 0%
\end{bmatrix}.
\end{equation*}%
We have%
\begin{eqnarray*}
U_{n} &=&K_{n}\mathbb{j}_{n}M_{n} \\
&=&\begin{bmatrix}
2 & y_{n}^{\prime } & -j_{n-1} & -j_{n-2} & \ldots & -j_{3} & -j_{2} \\
& y_{n} & \frac{1}{2}j_{n-1}-j_{0} & \frac{1}{2}j_{n-2}-j_{n-1} & \ldots &
\frac{1}{2}j_{3}-j_{4} & \frac{1}{2}j_{2}-j_{3} \\
&  & 1+2j_{n-1} & j_{n}-2 &  &  &  \\
&  &  & 1+2j_{n-1} &  &  &  \\
&  &  &  &  &  & 0 \\
&  &  &  &  &  &  \\
&  &  &  &  & j_{n}-2 &  \\
&  &  &  &  &  &  \\
&  &  &  &  & 1+2j_{n-1} & j_{n}-2 \\
&  &  &  &  &  & 1+2j_{n-1}%
\end{bmatrix}
\end{eqnarray*}%
{and}  $U_{n}$ is upper triangular, where
\begin{eqnarray*}
y_{n}&{:=}&\frac{1}{2}\left[ (4-j_{n-1})+\sum_{k=2}^{n-1}\left(
2j_{k}-j_{k-1}\right) \left( \frac{1+2j_{n-1}}{2-j_{n}}\right) ^{n-k}\right]{,}\\
y_{n}^{\prime }&{:=}&\sum_{k=1}^{n-1}j_{k}\left( \frac{1+2j_{n-1}}{2-j_{n}%
}\right) ^{n-k-1}.
\end{eqnarray*}%
Then we obtain%
\begin{equation*}
\det (U_{n})=\det (K_{n})\det (\mathbb{j}_{n})\det
(M_{n})=2(1+2j_{n-1})^{n-2}y_{n}.
\end{equation*}%
Since%
\begin{equation*}
\det (K_{n})=
\begin{cases}
1, & n\equiv 1\mbox{ or 2 } (\mbox{mod}4)\\
-1, & n\equiv 0\mbox{ or } 3\ (\mbox{mod}4){,}%
\end{cases}
\end{equation*}%
and%
\begin{equation*}
\det (M_{n})=
\begin{cases}
\ \ \left( \frac{1+2j_{n-1}}{2-j_{n}}\right) ^{n-2},&  n\equiv 1\mbox{ or 2 }
(\mbox{mod}4) \\
-\left( \frac{1+2j_{n-1}}{2-j_{n}}\right) ^{n-2},& n\equiv 0\mbox{ or 3 }  (\mbox
{mod}4){,}%
\end{cases}
\end{equation*}%
for all $n>3,$%
\begin{equation*}
\det (K_{n})\det (M_{n})=\left( \frac{1+2j_{n-1}}{2-j_{n}}\right) ^{n-2}
\end{equation*}%
{and} we have \eqref{60}.
\end{proof}

\section{ Inverses of {$\mathbb{J}_{n}$ and $\mathbb{j}_{n}$}}

{We will use the well-known fact that the inverse of a nonsingular circulant
matrix is also circulant [14, p.84] [12, p.33], [4, p.90-91].} 

\begin{theorem}
{The matrix} $\mathbb{J}_{n}=\mbox{circ}(J_{1},J_{2},\ldots ,J_{n})$ is
invertible {when} $n\geq 3$.
\end{theorem}

\begin{proof}
From Theorem \ref{t1},  $\det (\mathbb{J}_{3})=20\neq 0$ and $\det (\mathbb{J%
}_{4})=-400\neq 0.$ Then $\mathbb{J}_{3}$ and $\mathbb{J}_{4}$ are {invertible}. Let $n\geq 5.$ {The} Binet formula for Jacobsthal numbers
{gives} $J_{n}=\frac{\alpha ^{n}-\beta ^{n}}{3}${,} where $\alpha +\beta
=1,\alpha \beta =-2$ and $\alpha -\beta =3.$ Then we have%
\begin{eqnarray*}
g(\omega ^{k}) &=&\sum_{r=1}^{n}J_{r}\omega
^{kr-k}=\sum_{r=1}^{n}\left( \frac{\alpha ^{r}-\beta ^{r}}{3}\right)
\omega ^{kr-k}=\frac{1}{3}\sum_{r=1}^{n}\left( \alpha ^{r}-\beta
^{r}\right) \omega ^{kr-k} \\
&=&\frac{1}{3}\left[ \frac{\alpha (1-\alpha ^{n})}{1-\alpha \omega ^{k}}-%
\frac{\beta (1-\beta ^{n})}{1-\beta \omega ^{k}}\right] ,\ \ (1-\alpha
\omega ^{k},1-\beta \omega ^{k}\neq 0) \\
&=&\frac{1}{3}\left( \frac{(\alpha -\beta )-(\alpha ^{n+1}-\beta
^{n+1})+\alpha \beta \omega ^{k}(\alpha ^{n}-\beta ^{n})}{1-\alpha \omega
^{k}-\beta \omega ^{k}+\alpha \beta \omega ^{2k}}\right) \\
&=&\frac{1-J_{n+1}-2J_{n}\omega ^{k}}{1-\omega ^{k}-2\omega ^{2k}}{,}\qquad k=1,2,\ldots ,n-1.
\end{eqnarray*}%
 {If there existed} $\omega ^{k}$ $(k=1,2,\ldots ,n-1)$ such that $%
g(\omega ^{k})=0${, then we would have} $1-J_{n+1}-2J_{n}\omega ^{k}=0$ for $1-\omega ^{k}-2\omega ^{2k}\neq 0.$ Hence $\omega
^{k}=\frac{1-J_{n+1}}{2J_{n}}.$ It
is well known that%
\begin{equation}\label{omega}
\omega ^{k}=\exp \left( \frac{2k\pi i}{n}\right) =\cos \left( \frac{2k\pi }{n%
}\right) +i\sin \left( \frac{2k\pi }{n}\right)
\end{equation}
where $i:=\sqrt {-1}$. Since $\omega ^{k}=\frac{1-J_{n+1}}{2J_{n}}$
is a real number,  $\sin \left( \frac{2k\pi }{n}\right) =0$
{so that}  $\omega ^{k}=-1$ for $0<\frac{2k\pi }{n}<2\pi .$ However $%
u=-1$ is  not a root of the equation $1-J_{n+1}-2J_{n}u=0$ $(n\geq 5)$,
 {a} contradiction{,} i.e.{,} $g(\omega ^{k})\neq 0$ for any $
\omega ^{k}${, where} $k=1, 2, \dots , n-1$, $n\geq 5.$ Thus the proof is completed by
  [1, Lemma 1.1].
\end{proof}

\begin{lemma}
Let $A= (a_{ij})$ be {the} $(n-2)\times (n-2)$ matrix {defined by} 
\begin{equation*}
a_{ij}=%
\begin{cases}
2J_{n}, & i=j \\ 
J_{n+1}-1, & j=i+1 \\ 
0, & \mbox{otherwise}.%
\end{cases}%
\end{equation*}%
{Then} $A^{-1}=(a_{ij}^{^{\prime }})$ is {given by} 
\begin{equation*}
a_{ij}^{\prime }{:=}%
\begin{cases}
\frac{(1-J_{n+1})^{j-i}}{(2J_{n})^{j-i+1}}, & j\geq i \\ 
0, & \mbox{otherwise}.%
\end{cases}%
\end{equation*}
\end{lemma}

\begin{proof}
Let $B=(b_{ij})=AA^{-1}.$ Clearly 
 $b_{ij}=\sum_{k=1}^{n-2}a_{ik}a_{kj}^{^{\prime }}$.
{When} $i=j,$ we have%
\begin{equation*}
b_{ii}=2J_{n}{\cdot}\frac{1}{2J_{n}}=1.
\end{equation*}%
If {$j> i$,} then%
\begin{eqnarray*}
b_{ij} &=&\sum_{k=1}^{n-2}a_{ik}a_{kj}^{^{\prime
}}=a_{i,i+1}a_{i+1,j}^{^{\prime }}+a_{ii}a_{ij}^{^{\prime }} \\
&=&(J_{n+1}-1)\frac{(1-J_{n+1})^{j-i-1}}{(2J_{n})^{j-i}}+2J_{n}\frac{%
(1-J_{n+1})^{j-i}}{(2J_{n})^{j-i+1}}=0;
\end{eqnarray*}%
similar for $j<i$. Thus $AA^{-1}=I_{n-2}$. 
\end{proof}

\begin{theorem}
\bigskip Let the matrix $\mathbb{J}_{n}$ be $\mathbb{J}_{n}:=\mbox{circ}%
(J_{1},J_{2},\ldots ,J_{n})$ ($n\geq 3$). Then the inverse of the matrix $%
\mathbb{J}_{n}$ is%
\begin{equation*}
\mathbb{J}_{n}^{-1}=circ(m_{1},m_{2},\ldots ,m_{n})
\end{equation*}%
where%
\begin{eqnarray*}
m_{1} &=&\frac{J_{n+1}+(1-2J_{n-1})g_{n}-1}{2g_{n}J_{n}^{2}} \\
m_{2} &=&\frac{g_{n}-1}{J_{n}g_{n}} \\
m_{3} &=&\frac{1}{g_{n}}\left[ (1-J_{n}-g_{n})\frac{(1-J_{n+1})^{n-3}}{%
(2J_{n})^{n-2}}+2\dsum\limits_{k=2}^{n-2}J_{n-k}\frac{(1-J_{n+1})^{n-k-2}}{%
(2J_{n})^{n-k-1}}\right]
\end{eqnarray*}%
\begin{eqnarray*}
m_{4} &=&\frac{1}{g_{n}}\left( \frac{%
[(1-J_{n}-g_{n})(J_{n+2}-1)-4J_{n}J_{n-2}](1-J_{n+1})^{n-4}}{(2J_{n})^{n-2}}%
\right. \\
&&\left. +4\dsum\limits_{k=1}^{n-4}J_{k}\frac{(1-J_{n+1})^{k-1}}{(2J_{n})^{k}%
}\right) \\
m_{i} &=&\frac{1}{g_{n}}\frac{(1-J_{n+1})^{n-i}}{(2J_{n})^{n-i+1}}\left\{ 
\begin{array}{c}
\frac{(1-J_{n}-g_{n})(2^{n+2}-4)}{(2J_{n})^{2}}-2\left( J_{n-1}+\frac{%
J_{n-2}(1-J_{n+1})}{J_{n}}\right) ,\ n\ is\ odd \\ 
-2\left( J_{n-1}+\frac{J_{n-2}(1-J_{n+1})}{J_{n}}\right) ,\ \ \ \ \ \ \ \ \
\ \ \ \ \ \ \ \ \ \ \ \ \ \ \ \ \ \ n\ is\ even%
\end{array}%
\right.
\end{eqnarray*}%
for $g_{n}=1-J_{n}+2\tsum\nolimits_{k=1}^{n-2}J_{n-k-1}\left( \frac{2J_{n}}{%
1-J_{n+1}}\right) ^{k}$ and $i=5,6,\ldots ,n.$
\end{theorem}

\begin{proof}
Let%
\begin{equation*}
R_{n}=\left[
\begin{array}{cccccc}
1 & -f_{n} & \frac{f_{n}}{g_{n}}(J_{n}-1)+J_{n} & J_{n-1}-2\frac{f_{n}}{g_{n}%
}J_{n-2} & \ldots & J_{3}-2\frac{f_{n}}{g_{n}}J_{2} \\
0 & 1 & -\frac{J_{n}-1}{g_{n}} & \frac{2J_{n-2}}{g_{n}} & \ldots & \frac{%
2J_{2}}{g_{n}} \\
0 & 0 & 1 & 0 & \ldots & 0 \\
0 & 0 & 0 & 1 & \ldots & 0 \\
\vdots & \vdots & \vdots & \vdots & \ddots & \vdots \\
0 & 0 & 0 & 0 & \ldots & 0 \\
0 & 0 & 0 & 0 & \ldots & 1%
\end{array}%
\right]
\end{equation*}%
and $G=diag(1,g_{n})$ where $f_{n}=\tsum\nolimits_{k=1}^{n-1}J_{k+1}\left(
\frac{2J_{n}}{1-J_{n+1}}\right) ^{n-k-1}$ and $g_{n}=1-J_{n}+2\tsum%
\nolimits_{k=1}^{n-2}J_{n-k-1}\left( \frac{2J_{n}}{1-J_{n+1}}\right) ^{k}.$
Then we can write%
\begin{equation*}
P_{n}\mathbb{J}_{n}Q_{n}R_{n}=G\oplus A
\end{equation*}%
where $G\oplus A$ is the direct sum of the matrices $G$ and $A$. Let $%
T_{n}=Q_{n}R_{n}.$ Then we have%
\begin{equation*}
\mathbb{J}_{n}^{-1}=T_{n}(G^{-1}\oplus A^{-1})P_{n}.
\end{equation*}

Since the matrix $\mathbb{J}_{n}$ is circulant, its inverse is circulant
from Lemma 1.1 [1, p. 9791]. Let%
\begin{equation*}
\mathbb{J}_{n}^{-1}=circ(m_{1},m_{2},\ldots ,m_{n}).
\end{equation*}%
Since the last row of the matrix $T_{n}$ is
\begin{equation*}
\left( 0,1,\frac{1-J_{n}}{g_{n}}-1,\frac{2J_{n-2}}{g_{n}},\frac{2J_{n-3}}{%
g_{n}},\ldots ,\frac{2J_{3}}{g_{n}},\frac{2J_{2}}{g_{n}}\right) ,
\end{equation*}%
the last row entries of the matrix $\mathbb{J}_{n}^{-1}$ are%
\begin{eqnarray*}
m_{2} &=&\frac{g_{n}-1}{J_{n}g_{n}} \\
m_{3} &=&\frac{1}{g_{n}}\left( (1-J_{n}-g_{n})\frac{(1-J_{n+1})^{n-3}}{%
(2J_{n})^{n-2}}+2\dsum\limits_{k=2}^{n-2}J_{n-k}\frac{(1-J_{n+1})^{n-k-2}}{%
(2J_{n})^{n-k-1}}\right) \\
m_{4} &=&\frac{1}{g_{n}}\left( \frac{%
[(1-J_{n}-g_{n})(J_{n+2}-1)-4J_{n}J_{n-2}](1-J_{n+1})^{n-4}}{(2J_{n})^{n-2}}%
\right. \\
&&\left. +4\dsum\limits_{k=1}^{n-4}J_{k}\frac{(1-J_{n+1})^{k-1}}{(2J_{n})^{k}%
}\right)
\end{eqnarray*}%
\begin{eqnarray*}
m_{5} &=&\frac{1}{g_{n}}\left[ (1-J_{n}-g_{n})(2^{n+1}-2)(1+(-1)^{n-1})\frac{%
(1-J_{n+1})^{n-5}}{(2J_{n})^{n-2}}\right. \\
&&+\frac{2}{g_{n}}\left( \dsum\limits_{k=1}^{n-4}J_{k+3}\frac{%
(1-J_{n+1})^{k-1}}{(2J_{n})^{k}}-\dsum\limits_{k=1}^{n-3}J_{k+2}\frac{%
(1-J_{n+1})^{k-1}}{(2J_{n})^{k}}-2\dsum\limits_{k=1}^{n-2}J_{k+1}\frac{%
(1-J_{n+1})^{k-1}}{(2J_{n})^{k}}\right) \\
&&\vdots \\
m_{n} &=&\frac{1}{g_{n}}\left[ (1-J_{n}-g_{n})(2^{n+1}-2)(1+(-1)^{n-1})%
\right. \\
&&\left. +2J_{n-2}\left( -\frac{1}{2J_{n}}-2\frac{1-J_{n+1}}{(2J_{n})^{2}}%
\right) +2J_{n-3}\left( -\frac{2}{2J_{n}}\right) \right] \\
m_{1} &=&\frac{J_{n+1}+(1-2J_{n-1})g_{n}-1}{2g_{n}J_{n}^{2}}
\end{eqnarray*}%
where $g_{n}=1-J_{n}+2\tsum\nolimits_{k=1}^{n-2}J_{n-k-1}\left( \frac{2J_{n}%
}{1-J_{n+1}}\right) ^{k}.$ If we rearrange $m_{5},$ then%
\begin{eqnarray*}
m_{5} &=&\frac{1}{g_{n}}\left[ (1-J_{n}-g_{n})(2^{n+1}-2)(1+(-1)^{n-1})\frac{%
(1-J_{n+1})^{n-5}}{(2J_{n})^{n-2}}\right. \\
&&+\frac{2}{g_{n}}\frac{(1-J_{n+1})^{n-4}}{(2J_{n})^{n-3}}\left[ \left(
-J_{n-1}\frac{(1-J_{n+1})^{n-4}}{(2J_{n})^{n-3}}-2J_{n-1}\frac{%
(1-J_{n+1})^{n-3}}{(2J_{n})^{n-2}}-2J_{n-2}\frac{(1-J_{n+1})^{n-4}}{%
(2J_{n})^{n-3}}\right) \right. \\
&&\left. +\dsum\limits_{k=1}^{n-4}(\underset{0}{\underbrace{%
J_{k+3}-J_{k+2}-2J_{k+1}}})\frac{(1-J_{n+1})^{k-1}}{(2J_{n})^{k}}\right] \\
&=&\frac{1}{g_{n}}\frac{(1-J_{n+1})^{n-4}}{(2J_{n})^{n-3}}\left\{
\begin{array}{c}
\frac{(1-J_{n}-g_{n})(2^{n+2}-4)}{(2J_{n})^{2}}-2\left( J_{n-1}+\frac{%
J_{n-2}(1-J_{n+1})}{J_{n}}\right) ,\ n\ is\ odd \\
-2\left( J_{n-1}+\frac{J_{n-2}(1-J_{n+1})}{J_{n}}\right) ,\ \ \ \ \ \ \ \ \
\ \ \ \ \ \ \ \ \ \ \ \ \ \ \ \ \ \ n\ is\ even.%
\end{array}%
\right.
\end{eqnarray*}%
Then%
\begin{equation*}
m_{i}=\frac{1}{g_{n}}\frac{(1-J_{n+1})^{n-i}}{(2J_{n})^{n-i+1}}\left\{
\begin{array}{c}
\frac{(1-J_{n}-g_{n})(2^{n+2}-4)}{(2J_{n})^{2}}-2\left( J_{n-1}+\frac{%
J_{n-2}(1-J_{n+1})}{J_{n}}\right) ,\ n\ is\ odd \\
-2\left( J_{n-1}+\frac{J_{n-2}(1-J_{n+1})}{J_{n}}\right) ,\ \ \ \ \ \ \ \ \
\ \ \ \ \ \ \ \ \ \ \ \ \ \ \ \ \ \ n\ is\ even.%
\end{array}%
\right.
\end{equation*}%
and $i=5,6,\ldots ,n.\ $Since the matrix $\mathbb{J}_{n}^{-1}$ is a
circulant matrix and its last row is known, the proof is completed.
\end{proof}

{A Hankel matrix $A=(a_{ij})$ is an $n\times n$ matrix such that $a_{i,j} =
a_{i-1,j+1}$. It is closely related to the Toeplitz matrix in the sense that
a Hankel matrix is an upside-down Toeplitz matrix.}

\begin{corollary}
Let the matrix $P_{n}$ be as in (\ref{5}). Then 
\begin{equation*}
P_{n}^{-1}=\left[ 
\begin{array}{r}
H_{1} \\ 
H_{2}%
\end{array}%
\right]{,}
\end{equation*}%
where $H_{1}{:=}[1,0,\ldots ,0]$ and $H_{2}$ is an $(n-1)\times n$ Hankel
type {in which the} first row {is} $\left[ J_{n}, J_{n-1},\dots , J_{1}%
\right] $ and {the} last column {is} $\left[ 1, 0,\ldots , 0\right] ^{T}$.
\end{corollary}

\begin{proof}
The matrix $P_{n}^{-1}$ is computed easily by applying elementary row operations to the {augmented} matrix $[P_{n}:I_{n}].$
\end{proof}

\begin{theorem}
{The matrix} $\mathbb{j}_{n}=\mbox{circ}(j_{0},j_{1},\ldots ,j_{n-1})$ is
invertible when $n\geq 3$.
\end{theorem}

\begin{proof}
We show that $\det (\mathbb{j}_{3})=104\neq 0$ and $\det (\mathbb{j}%
_{4})=-675\neq 0$ by Theorem \ref{t2}. Then  $\mathbb{j}_{3}$ and $%
\mathbb{j}_{4}$ are invertible. Let $n\geq 5.$ {The} Binet formula for Jacobsthal-Lucas numbers {yields} $j_{n}=\alpha ^{n}+\beta ^{n}${,} where
$\alpha +\beta =1$ and $\alpha \beta =-2.$ Then we have%
\begin{eqnarray*}
h(\omega ^{k}) &=&\sum_{r=1}^{n}j_{r}\omega
^{kr-k}=\sum_{r=1}^{n}\left( \alpha ^{r}+\beta ^{r}\right) \omega
^{kr-k} \\
&=&\frac{\alpha (1-\alpha ^{n})}{1-\alpha \omega ^{k}}+\frac{\beta (1-\beta
^{n})}{1-\beta \omega ^{k}},\ \ (1-\alpha \omega ^{k},1-\beta \omega
^{k}\neq 0) \\
&=&\left( \frac{(\alpha +\beta )-(\alpha ^{n+1}+\beta ^{n+1})+\alpha \beta
\omega ^{k}(\alpha ^{n}+\beta ^{n})-2\alpha \beta \omega ^{k}}{1-\alpha
\omega ^{k}-\beta \omega ^{k}+\alpha \beta \omega ^{2k}}\right) \\
&=&\frac{1-j_{n+1}+2\omega ^{k}(2-j_{n})}{1-\omega ^{k}-2\omega ^{2k}}{,}\qquad  \ k=1,2,\ldots ,n-1.
\end{eqnarray*}%
We {are going to} show that there is {no} $\omega ^{k}$, $k=1,2,\ldots ,n-1$ such that $h(\omega ^{k})=0.$ If $1-j_{n+1}+2\omega
^{k}(2-j_{n})=0$
for $1-\omega ^{k}-2\omega ^{2k}\neq 0${, then} $\omega ^{k}=\frac{j_{n+1}-1}{%
2(2-j_{n})}$ would be a real number. By \eqref{omega}
we {would} have $\sin \left( \frac{2k\pi }{n}%
\right) =0$ {so that} $\omega ^{k}=-1$ for $0<\frac{2k\pi }{n}<2\pi .$
However $u=-1$  {is not} a root of the equation $1-j_{n+1}+2(2-j_{n})u=0$ $%
(n\geq 5)${,}  a contradiction. i.e., $h(\omega ^{k})\neq 0$ for any $\omega ^{k}$, {where} $k=1,2,\ldots ,n-1$ and $n\geq 5.$ Thus the proof is
completed by [1, Lemma 1.1].
\end{proof}

\begin{lemma}
If the matrix $\mathbb{S}= (s_{ij})_{i,j=1}^{n-2}$ is of the form%
\begin{equation*}
s_{ij}=%
\begin{cases}
1+2j_{n-1}, & i=j \\ 
j_{n}-2, & j=i+1 \\ 
0, & \mbox{otherwise}{,}%
\end{cases}%
\end{equation*}%
then $\mathbb{S}^{-1}=(s_{ij}^{^{\prime }})_{i,j=1}^{n-2}$ is {given by} 
\begin{equation*}
s_{ij}^{\prime } =%
\begin{cases}
\frac{(2-j_{n})^{j-i}}{(1+2j_{n-1})^{j-i+1}}, & j\geq i \\ 
0, & \mbox{otherwise.}%
\end{cases}%
\end{equation*}
\end{lemma}

\begin{proof}
Let $\mathbb{B}{:=\mathbb{SS}^{-1}}=(b_{ij})$ 
{so that} $b_{ij}=\sum_{k=1}^{n-2}s_{ik}s_{kj}^{^{%
\prime }}$. {Clearly }
\begin{equation*}
b_{ii}=(1+2j_{n-1})\cdot\frac{1}{1+2j_{n-1}}=1.
\end{equation*}%
If $j> i,$ then
\begin{eqnarray*}
b_{ij} &=&\sum_{k=1}^{n-2}s_{ik}s_{kj}^{^{\prime
}}=s_{i,i+1}s_{i+1,j}^{^{\prime }}+s_{ii}s_{ij}^{^{\prime }} \\
&=&(j_{n}-2)\frac{(2-j_{n})^{j-i-1}}{(1+2j_{n-1})^{j-i}}+(1+2j_{n-1})\frac{%
(2-j_{n})^{j-i}}{(1+2j_{n-1})^{j-i+1}}=0;
\end{eqnarray*}%
{similar for $j<i$.} Thus $\mathbb{SS}^{-1}=I_{n-2}$.
\end{proof}

\begin{theorem}
\textcolor[rgb]{0.93,0.08,0.17}{Let $n\geq 3$. The} inverse of the matrix $%
\mathbb{j}_{n}$ is%
\begin{equation*}
\mathbb{j}_{n}^{-1}=circ(h_{0},h_{1},\ldots ,h_{n-1})
\end{equation*}%
where%
\begin{eqnarray*}
h_{0} &=&\frac{1}{2y_{n}}\left( \frac{9j_{n}-18+(10-8j_{n-2})y_{n}}{%
(1+2j_{n-1})^{2}}\right) \\
h_{1} &=&\frac{1}{2y_{n}}\left( \frac{4y_{n}-9}{1+2j_{n-1}}\right) \\
h_{2} &=&\frac{1}{2y_{n}}\left[ (4-j_{n-1}-2y_{n})\frac{(2-j_{n})^{n-3}}{%
(1+2j_{n-1})^{n-2}}\right. \\
&&\left. +\dsum\limits_{k=2}^{n-2}(2j_{k+1}-j_{k})\left( \frac{%
(2-j_{n})^{k-2}}{(1+2j_{n-1})^{k-1}}\right) \right] \\
h_{3} &=&\frac{1}{2y_{n}}\left[ ((4-j_{n-1}-2y_{n})(j_{n+1}-1)\right. \\
&&-(2j_{n-1}-j_{n-2})(1+2j_{n-1}))\frac{(2-j_{n})^{n-4}}{(1+2j_{n-1})^{n-2}}
\\
&&\left. +2\dsum\limits_{k=1}^{n-4}\left( (2j_{k+1}-j_{k})\frac{%
(2-j_{n})^{k-1}}{(1+2j_{n-1})^{k}}\right) \right] \\
h_{i} &\textcolor[rgb]{0.93,0.08,0.17}{:=}&\frac{1}{2y_{n}}\left( \frac{%
(4-j_{n-1}-2y_{n})(j_{n+1}+8j_{n}-2j_{n-1}-9((-2)^{n}+1)}{(1+2j_{n-1})^{2}}%
\right. \\
&&\left. -2j_{n}+j_{n-1}-\frac{(4j_{n-1}-2j_{n-2})(2-j_{n})}{1+2j_{n-1}}%
\right) \frac{(2-j_{n})^{n-i-1}}{(1+2j_{n-1})^{n-i+2}}
\end{eqnarray*}%
for $y_{n}=\frac{1}{2}(4-j_{n-1})+\frac{1}{2}\tsum\nolimits_{k=2}^{n-1}%
\left( 2j_{k}-j_{k-1}\right) \left( \frac{1+2j_{n-1}}{2-j_{n}}\right) ^{n-k}$
and $i=4,5,\ldots ,n-1.$
\end{theorem}

\begin{proof}
Let%
\begin{eqnarray*}
Z_{n} &=&\left[
\begin{array}{cccc}
1 & -\frac{1}{2}y_{n}^{\prime } & \frac{y_{n}^{\prime }}{4y_{n}}%
(j_{n-1}-2j_{0})+\frac{1}{2}j_{n-1} & \frac{y_{n}^{\prime }}{4y_{n}}%
(j_{n-2}-2j_{n-1})+\frac{1}{2}j_{n-2} \\
0 & 1 & \frac{2j_{0}-j_{n-1}}{2y_{n}} & \frac{2j_{n-1}-j_{n-2}}{2y_{n}} \\
0 & 0 & 1 & 0 \\
0 & 0 & 0 & 1 \\
\vdots & \vdots & \vdots & \vdots \\
0 & 0 & 0 & 0 \\
0 & 0 & 0 & 0%
\end{array}%
\right. \\
&&\left. ~\ \ \ \ \ \ \ \ \ \ \ \ \ \ \ \ \ \ \ \ \ \ \ \ \ \ \ \ \ \ \ \ \
\ \ \ \ \ \ \ \
\begin{array}{cc}
\ldots & \frac{y_{n}^{\prime }}{4y_{n}}(j_{2}-2j_{3})+\frac{1}{2}j_{2} \\
\ldots & \frac{2j_{3}-j_{2}}{2y_{n}} \\
\ldots & 0 \\
\ldots & 0 \\
\ddots & \vdots \\
\ldots & 0 \\
\ldots & 1%
\end{array}%
\right]
\end{eqnarray*}%
and $\mathbb{G}=diag(2,y_{n})$ where $y_{n}=\frac{1}{2}(4-j_{n-1})+\frac{1}{2%
}\tsum\nolimits_{k=2}^{n-1}\left( 2j_{k}-j_{k-1}\right) \left( \frac{%
1+2j_{n-1}}{2-j_{n}}\right) ^{n-k}$ and $y_{n}^{\prime
}=\tsum\nolimits_{k=1}^{n-1}j_{k}\left( \frac{1+2j_{n-1}}{2-j_{n}}\right)
^{n-k-1}.$ Then we obtain%
\begin{equation*}
K_{n}\mathbb{j}_{n}M_{n}Z_{n}=\mathbb{G}\oplus \mathbb{S}
\end{equation*}%
where $\mathbb{G}\oplus \mathbb{S}$ is the direct sum of the matrices $%
\mathbb{G}$ and $\mathbb{S}$. If $\mathbb{T}_{n}=M_{n}Z_{n},$ then we have%
\begin{equation*}
\mathbb{j}_{n}^{-1}=\mathbb{T}_{n}(\mathbb{G}^{-1}\oplus \mathbb{S}%
^{-1})K_{n}.
\end{equation*}

Since the matrix $\mathbb{j}_{n}$ is circulant, the inverse matrix $\mathbb{j%
}_{n}^{-1}$\ is circulant from Lemma 1.1 [1, p. 9791]. Let%
\begin{equation*}
\mathbb{j}_{n}^{-1}=circ(h_{0},h_{1},\ldots ,h_{n-1}).
\end{equation*}%
Since the last row of the matrix $\mathbb{T}_{n}$ is
\begin{equation*}
\left( 0,1,\frac{2j_{0}-j_{n-1}}{2y_{n}}-1,\frac{2j_{n-1}-j_{n-2}}{2y_{n}},%
\frac{2j_{n-2}-j_{n-3}}{2y_{n}},\ldots ,\frac{2j_{4}-j_{3}}{2y_{n}},\frac{%
2j_{3}-j_{2}}{2y_{n}}\right) ,
\end{equation*}%
the last row elements of the matrix $\mathbb{j}_{n}^{-1}$ are%
\begin{eqnarray*}
h_{1} &=&\frac{1}{2y_{n}}\left( \frac{4y_{n}-9}{1+2j_{n-1}}\right)  \\
h_{2} &=&\frac{1}{2y_{n}}\left[ (4-j_{n-1}-2y_{n})\frac{(2-j_{n})^{n-3}}{%
(1+2j_{n-1})^{n-2}}\right.  \\
&&\left. +\dsum\limits_{k=2}^{n-2}(2j_{k+1}-j_{k})\left( \frac{%
(2-j_{n})^{k-2}}{(1+2j_{n-1})^{k-1}}\right) \right]
\end{eqnarray*}%
\begin{eqnarray*}
h_{3} &=&\frac{1}{2y_{n}}\left[
((4-j_{n-1}-2y_{n})(j_{n+1}-1)-(2j_{n-1}-j_{n-2})(1+2j_{n-1}))\frac{%
(2-j_{n})^{n-4}}{(1+2j_{n-1})^{n-2}}\right.  \\
&&\left. +2\dsum\limits_{k=1}^{n-4}\left( (2j_{k+1}-j_{k})\frac{%
(2-j_{n})^{k-1}}{(1+2j_{n-1})^{k}}\right) \right]  \\
h_{i} &=&\frac{1}{2y_{n}}\left( \frac{%
(4-j_{n-1}-2y_{n})(j_{n+1}+8j_{n}-2j_{n-1}-9((-2)^{n}+1)}{(1+2j_{n-1})^{2}}%
\right.  \\
&&\left. -2j_{n}+j_{n-1}-\frac{(4j_{n-1}-2j_{n-2})(2-j_{n})}{1+2j_{n-1}}%
\right) \frac{(2-j_{n})^{n-i-1}}{(1+2j_{n-1})^{n-i+2}} \\
h_{0} &=&\frac{1}{2y_{n}}\left( \frac{9j_{n}-18+(10-8j_{n-2})y_{n}}{%
(1+2j_{n-1})^{2}}\right)
\end{eqnarray*}%
where $y_{n}=\frac{1}{2}(4-j_{n-1})+\frac{1}{2}\tsum\nolimits_{k=2}^{n-1}%
\left( 2j_{k}-j_{k-1}\right) \left( \frac{1+2j_{n-1}}{2-j_{n}}\right) ^{n-k}$
and $i=4,5,\ldots ,n-1.\ $Since the matrix $\mathbb{j}_{n}^{-1}$ is a
circulant matrix and its last row is known, the proof is completed.
\end{proof}

\begin{corollary}
Let the matrix $K_{n}$ be as in (\ref{50}). Then {%
\begin{equation*}
K_{n}^{-1}:=%
\begin{bmatrix}
1 & 0 \\ 
C & D%
\end{bmatrix}%
{,}
\end{equation*}%
} where 
\begin{equation*}
C {:=}\left( \frac {j_{n-1}}{2}\ \frac{j_{n-2}}{2}\ \frac{j_{n-3}}{2} \
\ldots \frac{ j_{1}}{2}\right)_{(n-1)\times 1}^{T}
\end{equation*}
and $D$ is {the} $(n-1)\times (n-1)$ Hankel matrix {in which} the first row
is $\left[ J_{n-1},J_{n-2},\ldots ,J_{1}\right] $ and the last column is $%
\left[ J_{1},0,\ldots ,0\right] ^{T}.$
\end{corollary}

\begin{proof}
The matrix $K_{n}^{-1}$ is obtained easily by  applying elementary row operations to {the augmented matrix} $[K_{n}|I_{n}].$
\end{proof}

\end{document}